\documentclass{amsart}

\usepackage{amssymb}
\usepackage{amsfonts,cmtiup,comment,stmaryrd}

\newtheorem{theorem}{Theorem}[section]
\newtheorem{lemma}[theorem]{Lemma}
\newtheorem{proposition}[theorem]{Proposition}

\theoremstyle{remark}

\newcommand{\ed}{\end{document}}

\begin{document}

\title[On the length of finite groups and of fixed points]{On the length of finite groups\\ and of fixed points}

\author{E. I. Khukhro}
\address{Sobolev Institute of Mathematics, Novosibirsk, 630\,090,
Russia, \newline and University of Lincoln, Lincoln, U.K.}
\email{khukhro@yahoo.co.uk}
\thanks{This work was supported by CNPq-Brazil. The first author thanks  CNPq-Brazil and the University of Brasilia for support and hospitality that he enjoyed during his visits to Brasilia.}

\author{P. Shumyatsky}
\address{Department of Mathematics, University of Brasilia, DF~70910-900, Brazil}
\email{pavel@unb.br}


\begin{abstract}
The generalized Fitting height of a finite group $G$ is
the least number $h=h^*(G)$ such that $F^*_h(G)=G$, where the $F^*_i(G)$ is the generalized Fitting series: $F^*_1(G)=F^*(G)$ and $F^*_{i+1}(G)$ is the inverse image of $F^*(G/F^*_{i}(G))$. It is proved that if $G$ admits a soluble group of automorphisms $A$ of coprime order, then $h^*(G)$ is bounded in terms of $h^* (C_G(A))$, where $C_G(A)$ is the fixed-point subgroup, and the number of prime factors of $|A|$ counting multiplicities. The result follows from the special case when $A=\langle\varphi\rangle$ is of prime order, where it is proved that $F^*(C_G(\varphi ))\leqslant F^*_{9}(G)$.

The nonsoluble length $\lambda (G)$  of a finite group $G$ is defined as the minimum 
number of nonsoluble factors in a normal series each of whose factors either is soluble or is a direct product of nonabelian simple groups. It is proved that if $A$ is a group of automorphisms of $G$ of coprime order, then $\lambda (G)$ is bounded in terms of $\lambda (C_G(A))$  and the number of prime factors of $|A|$ counting multiplicities.
\end{abstract}

\maketitle

\section{Introduction}

The structure of an arbitrary finite group $G$ can be described in terms of length parameters related to normal series with `nice' sections. One way of constructing such a series is based on the generalized Fitting subgroup $F^*(G)$. Recall that $F^*(G)$ is the product of the Fitting subgroup $F(G)$ and all subnormal quasisimple subgroups; here a group is quasisimple if it is perfect and its  quotient by the centre is a non-abelian simple group. Then the \textit{generalized Fitting series} of $G$ is defined starting from  $F^*_1(G)=F^*(G)$, and then by induction, $F^*_{i+1}(G)$ being  the inverse image of $F^*(G/F^*_{i}(G))$. The least number $h$ such that $F^*_h(G)=G$ is naturally defined as the \textit{generalized Fitting height} $h^*(G) $ of $G$. Clearly, if $G$ is soluble, then $h^*(G)=h(G)$ is the ordinary Fitting height of $G$. Bounding the generalized Fitting height of a finite group $G$ greatly facilitates using the classification of finite simple groups (and is itself often obtained using the classification). One of such examples is the reduction of the Restricted Burnside Problem to soluble and nilpotent groups in  the Hall--Higman paper \cite{ha-hi}, where the generalized Fitting height was in effect bounded for groups of given exponent (using the classification as a conjecture at the time). A similar example is Wilson's reduction of the problem of local finiteness of periodic profinite groups to pro-$p$ groups in \cite{wil83}.

Another useful, more rough, length parameter is
the \textit{nonsoluble length} $\lambda (G)$  of a finite group $G$, which is defined as the minimum 
number of nonsoluble factors in a normal series each of whose factors either is soluble or is a direct product of nonabelian simple groups. More precisely, consider normal series
$$
1= G_0\leqslant G_1\leqslant \dots \leqslant G_{2h+1}=G
$$
in which  for $i$ even  the factor $G_{i+1}/G_{i}$ is soluble (possibly trivial), and for $i$ odd   the factor $G_{i+1}/G_{i}$   is a (non-empty) direct product of nonabelian simple groups; then the nonsoluble length $\lambda (G)$ is equal to the minimum possible value of $h$. (In particular, the group is soluble if and only if its nonsoluble length is $0$.) Bounding the nonsoluble length was one of the steps in Wilson's paper \cite{wil83}; more recently we used this parameter in the study of both finite and profinite groups in \cite{khu-shu131,khu-shu132}.

In the present paper we consider the  generalized Fitting height and nonsoluble length of a finite group $G$  relative to the same parameters of the fixed-point subgroup $C_G(A)$ of a  soluble group of automorphisms $A$ of coprime order. The results are modelled on Thompson's paper \cite{tho64}, where in the case of soluble groups $G$, $A$ of coprime orders, a bound for the Fitting height of $G$ was obtained in terms of the Fitting height of the fixed-point subgroup $C_G(A)$ and the number of prime factors of $|A|$ counting multiplicities.

\begin{theorem}\label{t1}
Suppose that a finite group $G$ admits a soluble group of automorphisms $A$ of coprime order. Then its generalized Fitting height  $h^* (G)$ is bounded in terms of the generalized Fitting height $h^*(C_G(A))$ of the fixed-point subgroup $C_G(A)$ and the number of prime factors of $|A|$ counting multiplicities.
\end{theorem}

In the proof we use the underlying result of Thompson \cite{tho64} for soluble groups in the special case when $A=\langle\varphi\rangle$ is of prime order: then  $F(C_G(\varphi ))\leqslant F_4(G)$. In fact, we also obtain an analogous result for  the generalized Fitting height.

\begin{theorem}\label{t2}
Suppose that a finite group $G$ admits an automorphism $\varphi $ of prime order coprime to $|G|$. Then $F^*(C_G(\varphi ))\leqslant F^*_{9}(G)$.
\end{theorem}

The proof of Theorem~\ref{t2} uses the following consequences of the classification:  the validity of Schreier's conjecture on solubility of outer automorphism groups of non-abelian finite simple groups and the result of Wang and Chen \cite{wa-ch}
that in a finite nonabelian simple group the fixed-point subgroup of an automorphism of
coprime order cannot be a nilpotent group.
Theorem~\ref{t1} follows from Theorem~\ref{t2} by a  straightforward induction, which furnishes the bound
$h^*(G) \leqslant 9^{\alpha (|A|)}\cdot h^*(C_G(A))+(9^{\alpha (|A|)}-1)/8$, 
where $|A|$ is the product of $ \alpha (|A|)$ primes.  We have no reason to believe that this bound is anywhere near a sharp one. It is worth mentioning that in the soluble case the original bound of Thompson $h(G) \leqslant 5^{\alpha (|A|)}\cdot h(C_G(A))$  (and $h(G) \leqslant 5^{\alpha (|A|)}$ if $ h(C_G(A))=0$ when $C_G(A)=1$) 
was later drastically improved by various authors, with first linear bound obtained by Kurzweil \cite{kur}, and definitive sharp bound by Turull \cite{tu}. We hope that similar improvements can be made for the bound in Theorem~\ref{t1}.

We now state a similar  result for the nonsoluble length.

\begin{theorem}\label{t3}
Suppose that a finite group $G$ admits a group of automorphisms $A$ of coprime order. Then its  nonsoluble length $\lambda (G)$ is bounded in terms of the nonsoluble length $\lambda (C_G(A))$ of the fixed-point subgroup $C_G(A)$ and the number of prime factors of $|A|$ counting multiplicities.
\end{theorem}

The proof of Theorem~\ref{t3} follows from the main case  where  $A=\langle \varphi\rangle$ is of prime order and $G=[G,\varphi ]$, where we prove that $\lambda (G)\leqslant \lambda (C_G(\varphi ))+1$. Straightforward induction then gives  the bound   $\lambda (G)\leqslant 2^{\alpha (|A|)}(\lambda (C_G(A))+1)-1$ in Theorem~\ref{t3}.
The proof uses the classification  in as much as the validity of Schreier's conjecture on solubility of outer automorphism groups of non-abelian finite simple groups.

The coprimeness conditions are unavoidable in all our theorems. For Theorems~\ref{t1} and \ref{t2} this was shown even for soluble groups by examples in Thompson's paper \cite{tho64}. For Theorem \ref{t3} one can take a repeated wreath product  $G=S\wr (S\wr  (\dots (S\wr S)\dots ))$ of a simple group $S$ containing an element $a$ of prime order with soluble centralizer $C_S(a)$. Let $b\in S$ be a conjugate of $a$ different from $a$. Then the element $(b,a,a,\dots ,a)$ in the base of the last wreath product has soluble  centralizer in $G$, while the nonsoluble length of $G$ is unbounded.

\medskip
Throughout the paper we use without special references the well-known
property of coprime actions:  if a group  $A$ acts by automorphisms on a finite
group $G$ of coprime order, $(|A|,|G|)=1$, then $ C_{G/N}(A)=C_G(A)N/N$ for any $A$-invariant normal subgroup $N$.
 We usually use the same letter to denote induced automorphisms of invariant sections.

\section{Generalized Fitting height}
\label{s-f}

In this section we obtain bounds for the generalized Fitting height in terms of the generalized Fitting height of the fixed-point subgroup of a group of automorphisms.  Recall that the generalized Fitting subgroup $F^*(G)$ of a finite group $G$ is the   product of the Fitting subgroup $F(G)$ and the characteristic subgroup $E(G)$, which is a central product of all subnormal quasisimple subgroups of $G$, that is, $E(G)=\prod Q_i$ over all $Q_i$ such that $Q_i$ is subnormal in $G$,  $Z(Q_i)\leqslant [Q_i,Q_i]$, and $S_i=Q_i/Z(Q_i)$ is a non-abelian simple group.  Then $[F(G),E]=1$ and $E(G)/Z(E(G))\cong F^*(G)/F(G)$ is the direct product of the $S_i$. Acting by conjugation, the group $G$ permutes the factors $Q_i$ and $C_G(F^*(G))\leqslant F(G)$. The following fact  (see, for example, \cite[Lemma~2.1]{khu-shu131}) is a well-known consequence of  Schreier's conjecture on solubility of outer automorphism groups of non-abelian finite simple groups confirmed by the classification.

\begin{lemma}\label{l21}
Let  $L/S(G)=F^*(G/S(G))$ be the generalized Fitting subgroup of the quotient by the soluble radical $S(G)$ of a finite group $G$, and let $K$ be the kernel of the permutational action of $G$ on the set of subnormal simple factors of $L/S(G)$. Then $K/L$ is soluble. \qed
\end{lemma}

The following elementary lemma may also be known.

\begin{lemma}\label{l1}
Let $N$ be a soluble normal subgroup of a finite group $G$. If $x\in F^*(G)$, then $x\in F(N\langle x\rangle )$.
\end{lemma}

\begin{proof}
 We claim that $[N\langle x\rangle,x,\dots ,x]=1$ if $x$ is repeated sufficiently many times. Write  $x=yz$, where $y\in F(G)$ and $z\in E(G)$. Since  $[N\langle x\rangle , x]\leqslant N\cap F^*(G)\leqslant F(G)$ and  $[F(G),z]=1$, we have
\begin{align*} [N,x,x,\dots ,x]&\leqslant [F(G), yz,\dots, yz ]\\
&=[F(G), y,\dots, y].
\end{align*}
The commutator on the right becomes trivial if it is sufficiently long, because $F(G)$ is nilpotent. As a result, $x\in F(N\langle x\rangle )$ by Baer's theorem (see \cite[Satz~III.6.15]{hup}).
\end{proof}

We shall use without special references the following well-known properties of the generalized Fitting subgroups relative to normal subgroups: if $N \trianglelefteqslant G$, then $F^*(N)=F^*(G)\cap N$ and $F^*(G)N/N \leqslant F^*(G/N)$. These and other properties follow, for example, from the fact that $F^*(G)$ is the set of all elements of $G$ that induce inner automorphisms on every chief factor of $G$; see, for example, \cite[Ch.~X, \S\,13]{hup3}. It is easy to see that similar properties hold for the higher terms of the generalized Fitting series:  if $N\trianglelefteqslant G$, then $F^*_i(N)=F^*_i(G)\cap N$ and $F^*_i(G)N/N \leqslant F^*_i(G/N)$.

\begin{proof}[Proof of Theorem~\ref{t2}] We have a finite group $G$ admitting  an automorphism $\varphi$ of prime order $p$ coprime to $|G|$. We need to show that $F^*(C_G(\varphi ))\leqslant F^*_{9}(G)$.  Recall that $S(G)$ denotes the soluble radical of $G$. Let $L$ be the inverse image of $F^*(G/S(G))$ and let  $L/S(G)=S_1\times \dots \times S_m$, where the $S_i$ are non-abelian simple groups.  Let $K$ denote the kernel of the permutational action of $G$ on  $\{S_1, \dots , S_m\}$. The quotient
 $K/L$ is soluble by Lemma~\ref{l21}.

\begin{lemma}\label{lh1} We have the inclusion
  $F^*(C_G(\varphi ))\leqslant K$.
\end{lemma}

\begin{proof}
We argue by contradiction. Let $x\in F^*(C_G(\varphi ))\setminus K$. Consider the quotient  $\bar G=G/S(G)$. Since $C_{\bar G}(\varphi )$ is the image of $C_{ G}(\varphi )$, then also $\bar x\in F^*(C_{\bar G}(\varphi ))\setminus \bar K$. Thus we can assume that $S(G)=1$ and  $L=S_1\times \dots \times S_m$. Then $x$ must permute these factors nontrivially.
Since $x\in C_G(\varphi )$, the element $x$ also permutes the orbits of $\varphi$ in the permutational action on $\{S_1, \dots , S_m\}$. If $x$ really moves some nontrivial $\varphi$-orbit, say, $\{S_{i_1}, \dots , S_{i_p}\}$, then $x$ moves the `diagonal' non-abelian simple group $D= C_G(\varphi )\cap (S_{i_1}\times  \dots \times  S_{i_p})$. Since $S_{i_1}\times  \dots \times  S_{i_p}$ is a subnormal subgroup, $D$ is subnormal in $C_G(\varphi )$ and therefore is normal in $F^*(C_G(\varphi ))$, contrary to being moved by $x$.

Thus,  $x$ stabilizes every nontrivial $\varphi$-orbit  $\{S_{i_1}, \dots , S_{i_p}\}$. Since the centralizer of $\varphi$ as a cycle of length $p$ in the symmetric group on $p$ symbols is $\langle\varphi\rangle$ and $p$ is coprime to $|G|$, it follows that $x$ must leave each factor $S_{i_1}, \dots , S_{i_p}$ fixed.

Therefore $x$ must really move some of the one-element orbits of $\varphi$, say, $S_i^x=S_j\ne S_i$, where both $S_i$ and $S_j$ are $\varphi$-invariant.
By the result of Wang and Chen \cite{wa-ch} based on the classification, $C_{S_i}(\varphi )$ cannot be a nilpotent group. If $x\in F(C_G(\varphi ))$, then we obtain a contradiction, since $[C_{S_i}(\varphi ),x]\leqslant F(C_G(\varphi ))$ 
and yet the projection of   $[C_{S_i}(\varphi ),x]$ onto $S_i$ 
covers $C_{S_i}(\varphi )$, which is not nilpotent. It remains to consider the case where $x\in F^*(C_G(\varphi ))\setminus F(C_G(\varphi ))$. Then at least one of the quasisimple components $Q$ of $F^*(C_G(\varphi ))$ is not contained in $K$ and moves $S_i$. Let $M$ be a minimal normal subgroup of $C_{S_i}(\varphi )$. Since $C_{S_i}(\varphi )$ is subnormal in $C_{G}(\varphi )$, it follows that $M\leqslant F^*(C_G(\varphi ))$. If $M$ is nilpotent, then $[M,Q]=1$. If $M$ is nonsoluble, then it is a product of quasisimple components  of $F^*(C_G(\varphi ))$ which are all different from $Q\not\leqslant L$ and then, too, $[M,Q]=1$. In either case we obtained a contradiction with $[M,Q]$ having nontrivial projection onto~$S_i$.
\end{proof}

We now complete the proof of Theorem~\ref{t2}. Recall that $L/S(G)=F^*(G/S(G))$ and $K$ is the kernel of the permutational action of $G$ on the set of subnormal simple factors of $L/S(G)$. By Lemma~\ref{lh1} we know that $F^*(C_G(\varphi ))$ is contained in $K$. Therefore $F^*(C_G(\varphi ))\leqslant F^*(C_K(\varphi ))$. Since  $F^*_{9}(K)\leqslant F^*_{9}(G)$,  it is sufficient to prove that  $F^*(C_K(\varphi ))\leqslant F^*_{9}(K)$. Thus, we can assume from the outset that $G=K$.

Let $x\in F^*(C_G(\varphi ))$ and let $X=\langle x^G\rangle$ be the normal closure of $x$. It is sufficient to prove that $X\leqslant F^*_9(G)$. Consider the soluble subgroup $H=S(X)\langle x\rangle $. By Lemma~\ref{l1} applied to  $C_G(\varphi )$ with $N=C_{S(X)}(\varphi )$, we obtain that $$
x\in F(C_{S(X)}(\varphi )\langle x\rangle )=F(C_H(\varphi )).
 $$
 Then $x\in F_4(H)$ by Thompson's theorem \cite{tho64}.

Therefore, $[x, S(X)]\leqslant F_4(H)\cap S(X)\leqslant F_4(S(X))\leqslant F_4(G)$. In the quotient $\bar G=G/F_4(G)$ this means that $\bar x\in C_{\bar G}({S(\bar X)})$. Since $X=\langle x^G\rangle$, it follows that $S(\bar X)\leqslant Z(\bar X)$. Then $F^*(\bar X)=S(\bar X)(\bar L\cap \bar X)$ because $L/S(G)$ is a direct product of simple groups and $G/L$ is soluble. Therefore $\bar X/F^*(\bar X)$ is a soluble group. Since $F^*(\bar X)\leqslant F^*(\bar G)$, we obtain that the image $\tilde X$ of $X$ in $\tilde G=G/F^*_5(G)$ is a soluble group.

We now apply Thompson's theorem \cite{tho64} to $\tilde X$. Namely, since $\tilde x\in F^*(C_{\tilde X}(\varphi ))= F(C_{\tilde X}(\varphi ))$, we have $\tilde x\in F_4(\tilde X)$.  Then $\tilde X = F_4(\tilde X)\leqslant F_4(\tilde G)$. As a result, $X\leqslant F^*_9(G)$, as required.
\end{proof}

\begin{proof}[Proof of Theorem~\ref{t1}]
We have a finite group $G$ admitting a soluble group of automorphisms $A$ of coprime order. We are going to prove that the generalized Fitting height $h^*(G)$ satisfies the inequality $h^*(G) \leqslant 9^{\alpha }\cdot h^*(C_G(A))+(9^{\alpha }-1)/8$, where $|A|$ is a product of $ \alpha =\alpha (|A|)$ primes (counting multiplicities). We proceed by induction on $ \alpha (|A|)$. Let $h=h^*(C_G(A))$. When $ \alpha (|A|)=1$, it follows from Theorem~\ref{t2} that $C_G(A)\leqslant F^*_{9h}(G)$, and then $G/F^*_{9h}(G)$ is nilpotent by Thompson's theorem \cite{tho59}. Thus,  $h^*(G)  \leqslant 9   h+1$ 
For $ \alpha (|A|)>1$, let $A_0$ be a normal subgroup of prime index in $A$. Then $C_G(A_0)$ admits the group of automorphisms $A/A_0$ of prime order and $C_{C_G(A_0)} (A/A_0)=C_G(A)$. By the above, $h^*(C_G(A_0)) \leqslant 9h+1$. By induction applied to $A_0$ we obtain
\begin{align*} 
h^*(G) &\leqslant 9^{\alpha -1 } \cdot (9 h+1)+ (9^{\alpha -1}-1)/8\\
            &=9^{\alpha }\cdot h+(9^{\alpha }-1)/8,
\end{align*}
as required.
\end{proof}

\section{Nonsoluble length}

Recall that the nonsoluble length $\lambda (G)$  of a finite group $G$ is the minimum 
number of nonsoluble factors in a normal series
each of whose factors either is soluble or is a direct product of nonabelian simple groups.
Consider the `upper nonsoluble series' of $G$, which by definition starts from the soluble radical $M_1=S(G)$, then $L_1$ is the full inverse image of $F^*(G/M_1)$, and then by induction $M_k$ is the full inverse image of the soluble radical of $G/L_{k-1}$ and $L_k$ the full inverse image of $F^*(G/M_k)$. It is easy to show that $\lambda (G)=m$, where $m$ is the first positive integer such that $M_{m+1}=G$. In the normal series
\begin{equation}\label{e-riad}
1=L_0\leqslant M_1 <  L_1\leqslant M_2<  \dots \leqslant M_{m+1}=G
\end{equation}
each quotient $U_i=L_i/M_i$ is a (nontrivial) direct product of nonabelian simple groups, and each quotient $M_i/L_{i-1}$ is soluble (possibly trivial). Since  $L_i/M_i=F^*(G/M_i)$, the properties mentioned at the beginning of \S\,\ref{s-f} apply. In particular,
if we write one of those  nonsoluble quotients as a direct product $U_k=S_1\times \dots \times S_t$ of nonabelian simple groups $S_i$, then the set of these factors $S_i$ is unique, characterized as the set of subnormal simple subgroups of $G/M_k$. The group $G$ and its automorphisms permute these subnormal factors, and for brevity we simply speak of their orbits on $U_k$ meaning orbits in this permutational action. The subgroup $L_k/M_k$ contains its centralizer in $G/M_k$.  The kernel of the permutational action of $G/L_k$ on $\{S_1, \dots , S_t\}$ is soluble by Lemma~\ref{l21}. Therefore the inverse image of this kernel is contained in $M_{k+1}$. We shall routinely use this fact without special references.

We shall need the following technical lemma.

\begin{lemma}\label{nn}
Let $N$ be a normal subgroup of a finite group $G$ and $K/N$ a simple subnormal subgroup of $G/N$. Let $D$ be
a  subgroup of $G$ such that $K=DN$. Suppose that $D\leqslant L_j(G)N$ and  $D\not\leqslant M_j(G)N$. Then $[D,D^x]NM_j(G)=KM_j(G)$ for any $x\in L_j(G)N$.
\end{lemma}

\begin{proof}    We can pass to the quotient $G/M_j(G)$ and without loss of generality assume that $j=1$ and $G$ has no nontrivial normal soluble subgroups. It follows that $K=SN$, where $S$ is a subnormal simple subgroup of $G$ (contained in  $L_1(G)$). The group $N$ permutes the subnormal  simple factors of $L_1$ contained in $K$ and  normalizes $L_1\cap N$; therefore $N$ also normalizes $S$. Hence,  $K=S\times N$. Now it is clear that the projection of $D$ onto $S$ is equal to $S$ and hence so is that of $D^x$ for any $x\in L_1(G)N$. Hence the projection of  $[D,D^x]$ also equals $S$, as $S$ is non-abelian simple.
\end{proof}

\begin{proof}[Proof of Theorem~\ref{t3}] We have  a finite group $G$ admitting a group of automorphisms $A$ of coprime order. We wish to show that its  nonsoluble length $\lambda (G)$ is bounded in terms of $\lambda (C_G(A))$ and the number of prime factors of $|A|$ counting multiplicities. We can of course assume that $G$ is nonsoluble, whence $A$ then has odd order and therefore is soluble by the Feit--Thompson theorem. The result will follow from the case of $|A|$ being a prime by straightforward induction on $|A|$. Thus the bulk of the proof is the following proposition.

\begin{proposition}\label{p-p}
Suppose that a finite group $G$ admits an automorphism $\varphi$ of prime order $p$ coprime to $|G|$ such that $[G, \varphi ]=G$.  Then $\lambda (G)\leqslant \lambda (C(\varphi ))+1$.
\end{proposition}

Clearly, we can assume that $G$ is nonsoluble; in particular, $|\varphi |\geqslant 3$.

We consider the `upper nonsoluble series' for $C_G(\varphi )$ constructed in the same way as \eqref{e-riad} was constructed for $G$, with its terms denoted by
\begin{equation}\label{e-rc}
1=\lambda _0\leqslant \mu _1< \lambda _1\leqslant \mu _2<  \dots \leqslant \mu _{l+1}=C_G(\varphi ),
\end{equation}
so that the quotients   $\lambda _i/\mu _i$ are (nontrivial) direct products of nonabelian simple groups, and the quotients $\mu _i/\lambda _{i-1}$ (possibly trivial) are soluble. Our task is to show that the nonsoluble length $\lambda (G)$ of $G$, which is equal to $m$ in \eqref{e-riad}, is at most $\lambda (C_G(\varphi ))+1$, that is, at most $l+1$.

\begin{lemma}\label{l-action}
For every $k\leqslant m-1$, the  automorphism $\varphi$ has nontrivial orbits (of length $p=|\varphi |$) on the quotient $U_k$.
\end{lemma}

\begin{proof}
For $k\leqslant m-1$, the group $G/L_{k}$ is nonsoluble and therefore by Lemma~\ref{l21} acts nontrivially by permutations on the set of subnormal simple factors of $U_{k}$. Since $G=[G,\varphi ]$ by hypothesis, the  automorphism $\varphi$ also has nontrivial orbits (of course, of length $p=|\varphi |$) on the set of these factors. Indeed, if $\varphi$ belonged to the kernel of the permutational action of $G\langle \varphi \rangle$, then $[G,\varphi ]$ would also belong to this kernel, which is a normal subgroup.
\end{proof}

Consider one of the 
 nonsoluble quotients $U_k=L_k/M_k$ such that the automorphism $\varphi$  has a nontrivial orbit  $\{S_1,\dots ,S_p$\} of length $p=|\varphi |$ on the set of subnormal factors of $U_k$.   Let  $S=S_1\times \dots \times S_p$, and let  $D=C_S(\varphi )$ be the diagonal subgroup of $S$, which is a non-abelian simple group  (isomorphic to $S_1$).

\begin{lemma}\label{l-pro} There is an index $i$ such that for any  $\varphi$-invariant subgroup $F$ with $FM_k/M_k\geqslant D$ there is a subgroup $\hat D\leqslant C_F(\varphi )$ such that
\begin{itemize}
\item[(1)] \ $\hat D M_k/M_k=D$,
\item[(2)] \ $\hat D\leqslant \lambda _i M_k$,
\item[(3)] \ $\hat D\not\leqslant \mu  _i M_k$.
\end{itemize}
\end{lemma}

\begin{proof} Recall that since the action of $\varphi$ on $G$ is coprime, the fixed points of $\varphi$ in any $\varphi$-invariant section are images of fixed points of $\varphi$ in $G$. Let
$\tilde D$ be any subgroup of $C_F(\varphi )$ such that $\tilde DM_k/M_k=D$.

  Consider the quotient $\bar G=G/M_k$. The images $\bar \lambda _i$ and $\bar \mu _i$ of the terms of the series \eqref{e-rc}  form a normal series of $C_G(\varphi )M_k/M_k$ with $\bar \lambda _i /\bar \mu _i$ being direct products (possibly empty) of non-abelian simple groups, and  $\bar \mu _i/ \bar \lambda _{i-1}$ being soluble.
  Since $D=\tilde DM_k/M_k$ is  a non-abelian simple subgroup of $C_G(\varphi )M_k/M_k$,
 there is an index $i$ such that $D\leqslant \bar \lambda _i$ and $D\not \leqslant \bar \mu _i$. This means  that $\tilde D\leqslant \lambda _iM_k$ and $\tilde  D\not \leqslant \mu _iM_k$.
\end{proof}

We fix the notation $\hat D$ for a subgroup of $C_F(\varphi )$ constructed for a diagonal subgroup $D$ of a nontrivial $\varphi$-orbit in $U_k= L_k/M_k$  in accordance with  Lemma~\ref{l-pro} (for some $F$); we choose $\hat D$ to be minimal by inclusion subject to satisfying
Lemma~\ref{l-pro}. Note that then $\hat D$ has no nontrivial soluble quotients. We then say for brevity that $\hat D$ \textit{corresponds to $D$} (with or without specifying for which~$F$); these subgroups will also be denoted by $D_k=D$ and  $\hat D_k=\hat D$, where the index only indicates the level of the quotient $U_k= L_k/M_k$.  The subgroups $\hat D_k$ can be chosen in a special way described in the following lemma. Note that since $\hat D_k\leqslant C_G(\varphi )$, the subgroup $\hat D_k $ permutes the orbits of $\varphi$ in its permutational action on the set of simple subnormal factors of $U_{k-1}$.

\begin{lemma}\label{l-acts} Suppose that $k\geqslant 2$ and  $D_k$ is a   diagonal subgroup in a nontrivial $\varphi$-orbit on $U_k$. Then there exists a subgroup $\hat D_{k}$ corresponding to $D_k$ that does not stabilize some nontrivial $\varphi$-orbits on $U_{k-1}$.
\end{lemma}

\begin{proof}
Note that the automorphism $\varphi$ does have nontrivial orbits on $U_{k-1}$ by Lemma \ref{l-action}.

Recall that $D_k=C_S(\varphi )$ is the `diagonal' in some $S=S_1\times\dots\times S_p$, where $S_i^{\varphi }=S_{i+1\,({\rm mod}\,p)}$. Choose a minimal  $\varphi$-invariant
subgroup $\hat S$ such that
   $\hat SM_k/M_k=S$.
(Subgroups with these properties obviously do exist: say, the inverse image of $S$.)
 Let $\hat D_{k}\leqslant C_{\hat S}(\varphi )$ be the subgroup corresponding to $D_k$  for $F=\hat S$ in the sense of Lemma~\ref{l-pro}. We claim that $\hat D_{k}$  is a required subgroup.
We argue by contradiction: suppose that $\hat D_{k}$ stabilizes all nontrivial $\varphi$-orbits on $U_{k-1}$. Let $\{R_1,\dots ,R_p\}$ be a nontrivial $\varphi$-orbit on $U_{k-1}$. Since the centralizer of $\varphi$ as a cycle of length $p$ in the symmetric group on $p$ symbols is $\langle\varphi\rangle$ and $p$ is coprime to $|G|$, it follows that $\hat D_{k}$ must leave each factor $R_1,\dots,R_p$ fixed.

On the other hand, since the kernel of the permutational action on  $U_{k-1}$ is contained in $M_k$ by Lemma~\ref{l21}, there is an element $d\in \hat D_{k}\setminus M_k$ that moves some of the subnormal simple factors of $U_{k-1}$. Let $R_0$ be such a factor  that is fixed by $\varphi$ but moved by $d$.

In the permutational action of $\hat S \langle\varphi\rangle$ on the set of subnormal factors of $U_{k-1}$, we now focus on the action of  $\hat S \langle\varphi\rangle$ on  its orbit containing $R_0$. Since $S$ is the only proper normal subgroup of $S\langle\varphi\rangle$ and $d\not \in M_k$, the kernel of this transitive action is contained in $\hat S\cap M_k$; in particular, $\varphi$ acts nontrivially. Let $H$ be the stabilizer of $R_0$ in  $\hat S \langle\varphi\rangle$. We have $\varphi \in H$, but $H$ does not contain $\hat S$, as it does not contain $d\in \hat D_{k}$. Since $\varphi$ acts nontrivially, there are elements  $x\in \hat S$ such that  $\varphi\not\in H^x$.

We claim that
\begin{equation}\label{e-cover}
S\leqslant HM_k/M_k.
 \end{equation}
 Indeed, by our supposition, whenever $\varphi\not\in H^y$ for $y\in\hat S$, we must have $\hat D_{k}\leqslant H^y$. In other words, $[\varphi ,x]\not\in H$ implies $\hat D_{k}^x\leqslant H$. First assume that there is no $x$ with this property  whose image modulo $M_k$ is a nontrivial element of $S_1$. Then $[S_1,\varphi ]\leqslant \bar H$, where $\bar H=HM_k/M_{k}$. We obtain that modulo $M_{k}$, in the coordinates of $S_1\times\dots \times S_p$, all elements $(a,a^{-1},1,\dots)$ are in $\bar H$. (Henceforth in this paragraph dots denote $1$s.) Multiplying $(a,a^{-1},1,\dots)$ by $(b,b^{-1},1,\dots)$ and then by $((ab)^{-1},(ab),1,\dots)$  for some non-commuting elements $a,b\in S_1$ we obtain an element $(1,g,1,\dots)\in \bar H$ with $g\ne 1$. Then conjugating by $(1,a,a^{-1},\dots)\in \bar H$ we obtain $(1,S_2,1,\dots)\leqslant \bar H$. Here the element $(1,a,a^{-1},\dots)$ belongs to $\bar H$ because $(1,a,a^{-1},\dots)=(a,a^{-1},1,\dots)^{\varphi}\in H^{\varphi}=H$, as  $\varphi\in H$; also recall that $|\varphi |>2$. From $(1,S_2,1,\dots)\leqslant \bar H$ by the action of $\varphi$  we obtain $S\leqslant \bar H$.

Thus we can assume that there is $x\in \hat S$ whose image modulo $M_k$ is a nontrivial
element of $S_1$ such that $[\varphi ,x]\not\in H$ and therefore, $\hat D_{k}^x\leqslant H$. This means that modulo $M_{k}$, in the coordinates of $S_1\times\dots \times S_p$, all elements $(b^x,b,b,\dots)$ lie in $\bar H$ and there is $b$ with $b^x\ne b$. Using the action of $\varphi$ we obtain $(b,b^x,b,\dots)\in \bar H$, whence $(c,c^{-1},1,\dots)\in \bar H$ for $c=[b,x]\ne 1$. Conjugating by $(b,b,b^x,\dots)$ we obtain $(c^b,(c^{-1})^b,1,\dots)\in \bar H$ for any~$b$. We can choose $b_1,b_2\in S_1$ such that
$[c^{b_1},c^{b_2}]\ne 1$ and therefore also $[(c^{-1})^{b_1},c^{b_2}]\ne 1$. We conjugate $(c^{b_1},(c^{-1})^{b_1},1,\dots)$ by $(1,c^{b_2},\dots)$ and obtain $(c^{b_1},((c^{-1})^{b_1})^{c^{b_2}},1,\dots)\in \bar H$. Multiplying by the inverse of $(c^{b_1},(c^{-1})^{b_1},1,\dots)$ we obtain  $(1,g,1,1,\dots)\in \bar H$ with $g\ne 1$. Conjugating by $(b^x,b,b,\dots)$ for all $b$ we again  obtain $S_2\leqslant \bar H$, and then $S\leqslant \bar H$, completing the proof of \eqref{e-cover}.

Now consider $H\cap \hat S$, which is a $\varphi$-invariant subgroup. By \eqref{e-cover} above, $\hat S \leqslant HM_k$; we claim that also  $\hat S\leqslant (H\cap \hat S)M_k$. For any $s\in \hat S$ we have $s=ha$ for $h\in H$ and $a\in M_k$. But $H\leqslant \hat S\langle\varphi\rangle$, so $h=s_1\varphi ^i$ for some $i$ and $s_1\in \hat S$, 
and then $s=s_1\varphi ^ia$. Clearly, $\varphi ^i=1$, so $s=s_1a$, where $s_1 =h\in H\cap \hat S$. Finally,   $H\cap \hat S$ is a proper subgroup of $\hat S$, since $d\not\in H$.  This contradicts the choice of $\hat S$ as a minimal $\varphi$-invariant
subgroup  such that
   $\hat SM_k/M_k=S$.
\end{proof}

\begin{proof}[Proof of Proposition~\ref{p-p}] The idea is to construct a chain of subgroups $\hat D_{m-1}, \dots , \hat D_1$, working from the second highest quotient in \eqref{e-riad} downwards,  in such a way that each of them  `marks' a strictly higher quotient among the $\lambda _i/\mu _i$ in the series \eqref{e-rc} of  $C_G(\varphi )$ than the next one with smaller index.

The proposition is obviously correct if $m\leqslant1$. Therefore we assume that $m\geqslant2$. By Lemma \ref{l-action} the automorphism $\varphi$ has a nontrivial orbit on the set of subnormal factors of $L_{m-1}/M_{m-1}$. If $m\geqslant 3$, let $\hat D_{m-1}\leqslant C_G(\varphi )$ be a subgroup of level $m-1$ constructed  by Lemma \ref{l-acts} in such a way that $\hat D_{m-1}$ does not stabilize some nontrivial $\varphi$-orbit on $L_{m-2}/M_{m-2}$. Let $D_{m-2}$ be the diagonal in one of these orbits  which is not stabilized by $\hat D_{m-1}$. Then  let $\hat D_{m-2}\leqslant C_G(\varphi )$  be a subgroup of level $m-2$ corresponding to $D_{m-2}$ by Lemma~\ref{l-acts} that does not stabilize some nontrivial $\varphi$-orbit on $L_{m-3}/M_{m-3}$. Continuing in this manner, we obtain a sequence of subgroups $\hat D_{m-1},\dots, \hat D_1$ such that for every $k$ the subgroup $\hat D_{k}$ is constructed by Lemma~\ref{l-acts} from the diagonal $D_k$ of a nontrivial $\varphi$-orbit on the subnormal factors of $L_{k}/M_{k}$, which is not stabilized by $\hat D_{k+1}$. Therefore for every $k\geqslant 2$ there exists  $x\in\hat D_{k}$ such that $D_{k-1}^x$ is the `diagonal' in a different orbit from the orbit where $ D_{k-1}$ is the `diagonal'. Then of course,
  \begin{equation}\label{ec}
    [\hat D_{k-1},\hat D_{k-1}^x]\leqslant M_{k-1} .
\end{equation}

We claim that $\hat D_{k}$ is not contained in $\mu_kM_k$ for $k=1,\dots,m-1$. Clearly,  $\hat D_1\not\leqslant\mu_1M_1$, since $\hat D_1$ is not soluble. By induction suppose that $\hat D_{k-1}\not\leqslant\mu_{k-1}M_{k-1}$. Let $j$ be the index given by Lemma~\ref{l-pro} such that $\hat D_{k-1}\leqslant\lambda_{j}M_{k-1}$ and $\hat D_{k-1}\not\leqslant \mu_{j}M_{k-1}$. Clearly,  $k-1\leqslant j$. Let $N=C_G(\varphi )\cap M_{k-1}$; then also
\begin{equation}\label{e1}
\hat D_{k-1}\not\leqslant \mu_{j}N.
\end{equation}

To perform the induction step we argue by contradiction and suppose that $\hat D_{k}\leqslant \mu_kM_k$. Since $\mu _k/\lambda _{k-1}$ is soluble and  $  \hat D_{k}M_{k}/M_{k}=D_k$ is non-abelian simple,   we must have $\hat D_{k}\leqslant \lambda_{k-1}M_k$.  Since $k-1\leqslant j$, it follows that $\hat D_{k}\leqslant \lambda_{j}M_k$ and therefore $\hat D_k\leqslant \lambda _j(M_k\cap C_G(\varphi ))$.

The group $C_G(\varphi )$ acts on the set of simple factors of the image of $\lambda_j/\mu_j $ in $C_G(\varphi )/N$.  By the above, the image $D_{k-1}$ of $\hat D_{k-1}$ is one of these factors. Clearly, $\lambda _j$ is in the kernel of this action by
Lemma~\ref{nn}. 
 It is easy to see that the image of $L_{k-1}\cap C_G(\varphi )$ stabilizes $D_{k-1}$ and therefore as a normal subgroup is in the kernel of the action on the orbit containing $D_{k-1}$. Since $M_k/L_{k-1}$ is soluble, we obtain that the image of $\lambda_j(M_k\cap C_G(\varphi ))$ in this action on the orbit containing $D_{k-1}$ is a soluble group. But $\hat D_k\leqslant \lambda_j(M_k\cap C_G(\varphi ))$, so the image of $\hat D_k$ is soluble. By minimality this image is actually trivial. Then
  \eqref{ec} 
   implies $\hat D_{k-1}\leqslant \mu_jN$, in contradiction
   with~\eqref{e1}. 
   
Thus, indeed $\hat D_{k}$ is not contained in $\mu_kM_k$ for $k=1,\dots,m-1$.
In particular, $\hat D_{m-1}$ is not contained in $\mu_{m-1}M_{m-1}$ and therefore $\lambda_{m-1}\ne\mu_{m-1}$, whence $\lambda (C_G(\varphi ))\geqslant m-1$.
\end{proof}

We now complete the proof of Theorem~\ref{t3}. Actually we prove  by
induction on $\alpha 
 =\alpha (|A|)$ that  $\lambda (G)\leqslant 2^{\alpha }(\lambda (C_G(A))+1)-1$. Let $\lambda=\lambda (C_G(A))$ for brevity.  When $\alpha (|A|)=1$, that is, $A=\langle \varphi\rangle$ is of prime order, then  $\lambda ([G,\varphi ])\leqslant \lambda +1$ by Proposition~\ref{p-p}. Then $\lambda (G)\leqslant 2\lambda +1$, since $G=[G, \varphi ]C_G(\varphi )$ and $[G, \varphi ]$ is a normal subgroup. 
For  $\alpha (|A|)>1$, let
$A_0$ be a normal subgroup of prime index in $A$. Then $C_G(A_0)$ admits the group of automorphisms $A/A_0$ of prime order and $C_{C_G(A_0)} (A/A_0)=C_G(A)$. By the above, $\lambda (C_G(A_0)) \leqslant 2\lambda +1$. It remains to apply the induction hypothesis to $A_0$:
\begin{align*}
\lambda (G)&\leqslant 2^{\alpha -1}(\lambda (C_G(A_0))+1)-1\leqslant 2^{\alpha -1 }(2\lambda +1+1)-1\\&=2^{\alpha -1 }\cdot 2\lambda +2^{\alpha -1}\cdot 2-1=2^{\alpha }(\lambda +1)-1. \end{align*}
\end{proof}

\section*{Acknowledgement}
The authors thank the referee for careful reading and helpful comments.

\end{document}